\documentclass{amsart}
\usepackage{amsmath,amssymb,hyperref,mathrsfs,graphicx,bbm}
\usepackage[utf8x]{inputenc}
\usepackage{amsmath}
\usepackage{amsfonts}
\usepackage{latexsym}
\usepackage{amsthm}
\usepackage{stmaryrd,amscd}
\usepackage{xargs}
\usepackage{tikz}
\usepackage{graphicx}
\usepackage{color}
\usepackage{enumitem}
\usepackage{graphicx}
\usepackage{color}

\usepackage[colorinlistoftodos,prependcaption]{todonotes}
\presetkeys%
{todonotes}%
{inline,backgroundcolor=white}{}
\tikzset{/tikz/notestyleraw/.append style={text=black}}

\newtheorem{thm}{Theorem}[section]

\newtheorem{lem}[thm]{Lemma}
\newtheorem{defn}[thm]{Definition}
\newtheorem{prop}[thm]{Proposition}

\newtheorem{rmk}[thm]{Remark}

\newtheorem*{thmA}{Theorem A}
\newtheorem*{thmB}{Theorem B}
\newtheorem*{thmC}{Theorem C}
\newtheorem*{thmD}{Theorem D}
\newtheorem*{thmE}{Theorem E}
\newtheorem*{corF}{Corollary F}

\newcommand{\be}{\begin{eqnarray}}
\newcommand{\ee}{\end{eqnarray}}
\newcommand{\beq}{\begin{equation}}
\newcommand{\eeq}{\end{equation}}
\newcommand{\ben}{\begin{eqnarray*}}
\newcommand{\een}{\end{eqnarray*}}
\newcommand{\beal}{\begin{aligned}}
\newcommand{\enal}{\end{aligned}}

\newcommand{\gm}{\gamma}

\newcommand{\lb}{\lambda}

\newcommand{\N}{\mathbb{N}}
\newcommand{\cC}{\mathcal{C}}

\newcommand{\R}{\mathbb{R}}

\newcommand{\om}{\omega}
\newcommand{\Om}{\Omega}

\newcommand{\dt}{\delta}

\newcommand{\cS}{\mathcal{S}}

\newcommand{\cA}{\mathcal{A}}

\newcommand{\cK}{\mathcal{K}}

\newcommand{\cU}{\mathcal{U}}

\newcommand{\cT}{\mathcal{T}}
\newcommand{\wh}{\widehat }
\newcommand{\wt}{\widetilde }

\title{Smooth subsolutions of the discounted Hamilton-Jacobi equations}
\thanks{$\dagger$ Corresponding author}
\date{}
\begin{document}
\maketitle

\centerline{\scshape Xiyao Huang}
\medskip
{\footnotesize
\centerline{ School of Mathematics and Statistics,}
 \centerline{Nanjing University of Sciences and Technology, Nanjing 210094, China}
 \centerline{{\it Email:} 123130011633@njust.edu.cn}
}
\bigskip

\centerline{\scshape Liang Jin  $^\dagger$}
\medskip
{\footnotesize
\centerline{  School of Mathematics and Statistics,}
 \centerline{Nanjing University of Sciences and Technology, Nanjing 210094, China}
 \centerline{{\it Email:} jl@njust.edu.cn}
}
\bigskip

\centerline{\scshape Jianlu Zhang}
\medskip
{\footnotesize
\centerline{Hua Loo-Keng Key Laboratory of Mathematics \&}
 \centerline{Mathematics Institute, Academy of Mathematics and systems science}
 \centerline{Chinese Academy of Sciences, Beijing 100190, China}
  \centerline{{\it Email:} jellychung1987@gmail.com}
}
\bigskip

\centerline{\scshape Kai Zhao}
\medskip
{\footnotesize
\centerline{School of Mathematical Sciences,}
\centerline{ Key Laboratory of Intelligent Computing and Applications(Ministry of Education),}
\centerline{Tongji University, Shanghai 200092, China;}
  \centerline{{\it Email:}  zhaokai93@tongji.edu.cn}
}
\bigskip

\begin{abstract}
For the discounted Hamilton-Jacobi equation 
\[
\lb u+H(x,d_{x}u)=0,\quad\lb>0
\]
 on compact, boundless manifold $M$, we construct $C^{1,1}$ subsolutions which are exactly solutions on the projected Aubry set. The smoothness of such subsolutions can be improved if additional hyperbolicity of the Aubry set is assumed. As applications, 
such subsolutions can be used to identify the maximal global attractor of the associated conformally symplectic flow, and to control the convergent speed of the Lax-Oleinik semigroup.

\end{abstract}

\vspace{20pt}


\section{Introduction}\label{s1}
\vspace{20pt}


Let $M$ be a smooth, closed Riemannian manifold. A function $H\in C^{k}(T^*M,\R)$ ($k\geq 2$) is called a {\sf Tonelli Hamiltonian} if for all $x\in M$,
\begin{enumerate}
  \item[\bf (H1)] {\sf (Positive Definite)} $H_{pp}(x,p)$ is positive definite for all $(x,p)\in T^*M$;
  \item[\bf (H2)] {\sf (Superlinear)} $\lim_{|p|_{x}\rightarrow+\infty}H(x,p)/|p|_{x}=+\infty$, where $|\cdot|_{x}$ is the norm on $T^{\ast}_{x}M$ induced by the Riemannian metric.
\end{enumerate}

For a fixed constant $\lb>0$, we consider the ODE system on $T^{\ast}M$ associated with $H$, which can be expressed in coordinates as
\be\label{eq:ode0}
\left\{
\begin{aligned}
\dot{x}&=H_p(x,p),\\
\dot p&=-H_x(x,p)-\lb p.
\end{aligned}
\right.
\ee
Physically, this system describes the mechanical motion of masses suffering friction proportional to the velocity. The earliest research of system \eqref{eq:ode0} can trace back to Duffing's work on explosion engines \cite{D}. Systems of the form \eqref{eq:ode0} can also be
found in other subjects, e.g. astronomy \cite{CC}, transportation \cite{WL} and economics \cite{B}.  

\medskip

{Due to (H1)-(H2)}, the {\sf local phase flow} $\Phi_{H,\lb}^t$ of \eqref{eq:ode0} is forward complete, namely, it is well defined for all $t\in\R_+$. Besides, the direct computation shows that $\Phi_{H,\lb}^t$ transports the standard symplectic form $\Omega=dp\wedge dx$ into a multiple of itself:
\be\label{cs}
(\Phi^{t}_{H,\lb})^{*}\Omega= e^{\lb t}\Omega,\quad t\in\R_+.
\ee
That is why \eqref{eq:ode0} is also called {\sf conformally symplectic} \cite{WL} or {\sf dissipative} \cite{LC} in some literatures. We also mention that more general dissipative twist maps have been studied by Le Calvez \cite{LC} and Casdagli \cite{Ca}. Besides, a KAM iteration for \eqref{eq:ode0} has been established in \cite{CCD}.

Recently, variational significance was exploited \cite{CCJWY2019,CCWY2018,WWY2,WWY,WWY3,WWY4} for generalized $1^{st}$ order PDEs. These works share the same dynamic taste with our paper.

\subsection{Viscosity subsolutions of discounted H-J equations}
 In \cite{DFIZ,MS}, the authors used a variational principle to get the
{\sf viscosity solution} $u^-(x)\in C(M,\R)$ for the {\sf discounted Hamilton-Jacobi equation}
\begin{align}\label{eq:hj}
\lb u+H(x,du)=0,\quad\tag{D}\quad x\in M
\end{align}
and showed its optimality in deciding minimal curves. Due to the {\sf Comparison Principle}, such a viscosity solution is unique but usually not $C^{1}$. Precisely, we define the {\sf Legendre transformation} by
\[
\mathcal{L}_H:T^{\ast}M\rightarrow TM;(x,p)\mapsto(x,H_p(x,p))
\]
which is a diffeomorphism due to  (H1)-(H2). Accordingly, 
\be\label{eq:led}
L(x,v):=\max_{p\in T_{x}^*M}\big\{\langle p, v\rangle-H(x,p)\big\},
\ee
defines a {\sf Tonelli Lagrangian} function in $ C^{k}(TM,\R)$, where the maximum is attained at $\bar p\in T_x^*M $ such that $\bar p=L_v(x, v)$.

\begin{defn}\label{defn:subsol}
A function $u\in C(M,\R)$ is called {\sf $\lb-$dominated by $L$} and denoted by $u\prec_\lb L$, 
if for any absolutely continuous $\gm:[a,b]\rightarrow M$, there holds
\be\label{eq:dom}
e^{\lb b}u(\gm(b))-e^{\lb a}u(\gm(a))\leq \int_a^be^{\lb t}L(\gamma(t),\dot\gamma(t))\ dt.
\ee
We can denote by $\cS^-$ the set of all $\lb-$dominated functions of $L$.
\end{defn}
\begin{rmk}
Recall that any $\lb-$dominated function $u$ has to be Lipschitz (Proposition 6.3 of \cite{DFIZ}). By the {\sf Rademacher's theorem}
\[
\lb u(x)+H(x,du(x))\leq 0,\quad a.e. \  x\in M,
\]
which implies $u$ is an {\sf almost everywhere subsolution} of  (\ref{eq:hj}). On the other side, the equivalence between almost everywhere subsolution and {\sf viscosity subsolution} was proved in a bunch of references e.g. \cite{BC,Ba,F2,Si}. So we get the equivalence among the three:
\[
\text{a.e. subsolution} \Longleftrightarrow \text{viscosity subsolution} \Longleftrightarrow \lb-\text{dominated function}
\]
\end{rmk}
\smallskip

%

\begin{defn}\label{def:Aubry}
$\gm\in C^{ac}(\R,M)$ is called {{\sf globally calibrated}} by $u\in\cS^-$, if for any $a<b\in\R$, 
\[
e^{\lb b}u(\gm(b))-e^{\lb a} u(\gm(a))=\int_a^be^{\lb t}L(\gamma(t),\dot\gamma(t))dt.
\]
The {\sf Aubry set} $\wt \cA$ is defined by
\be\label{defn:aub}
\wt\cA=\bigcup_{u\in\cS^-}\bigcup_{\gamma\in C^{ac}(\R,M)}\,\,\{(\gamma, \dot\gamma)|\gamma \text{ is globally calibrated by }u\}\subset TM
\ee
and the {\sf projected Aubry set} can be defined by $\cA=\pi\wt\cA\subset M$, where $\pi:TM\rightarrow M$ is the canonical projection.
\end{defn}

\begin{rmk} 


As described in Theorem (ii) of \cite{MS}, $\wt\cA$ is nonempty and $\Phi_{L,\lb}^t-$invariant. Moreover, $\pi^{-1}:\cA\rightarrow \wt\cA\subset TM$ is a Lipschitz graph.
\end{rmk}

The relation between $u^-$ and $\cS^-$ can be revealed by the following conclusion:

\begin{thm}[proved in Appendix \ref{a2}]\label{thm:0}
The viscosity solution of (\ref{eq:hj}) is the pointwise supreme of all smooth, i.e., $C^{\infty}$, viscosity subsolutions, namely we have
\be
u^-(x)=\sup_{u\in \cS^-}u(x)=\sup_{u\in C^{\infty}\cap\cS^-}u(x).
\ee
\end{thm}
\medskip

\subsection{Constrained subsolutions \& Main results} 

Notice that any $w\in C^{1}(M,\R)$ far smaller than $0$ can be a subsolution of \eqref{eq:hj}. Therefore, much of the elements in $\cS^-$ can't tell us any information about $\wt\cA$. For this purpose, 
we need a further selection in $\cS^-$. \smallskip

Let us denote
$\mathfrak M_\lb$ by the set of all $\Phi_{L,\lb}^t$-invariant measures (see \eqref{eq:e-l} for the definition). It is non-empty, since 
 $\wt\cA$ always supports at least one invariant probability measure, due to the {\sf Krylov-Bogolyubov's theorem}. That gives us a chance to define a special kind of subsolutions:
\begin{defn}\label{con-sol}
$u\in\cS^-$ is called a {\sf constrained subsolution} of \eqref{eq:hj}, if
\be\label{eq:con}
\inf_{\mu\in\mathfrak M_\lb}\int L-\lb u\,\,d\mu=0.
\ee
We denote by $\cS_c^- \subset \cS^-$ the set of constrained subsolutions.
\end{defn}
Our conclusion reveals fine properties of the constrained subsolutions:

\begin{thmA}\label{prop:cs}
\
\begin{enumerate}
\item $u^-\in\cS_c^-$, which implies $\cS_c^-\neq\emptyset$.
\item For each $u\in\cS_c^-$, there exists an $\Phi_{L,\lb}^t-$invariant subset $\wt\cA(u)\subset\wt\cA$, such that $u=u^-$ on $\pi\wt\cA(u)$.
\end{enumerate}
\end{thmA}

Moreover, such constrained subsolutions can be made differentiable:

\begin{thmB}\label{thm:b}
There exists a sequence of $u_{n\in\N}\in C^{1,1}(M,\R)\cap\cS_c^-$ which is exactly a solution on the projected Aubry set $\cA$ and converges to $u^-$ as $n\rightarrow+\infty$ w.r.t. the $C^0-$norm.
\end{thmB}

The smoothness of constrained subsolutions can be further improved, if additional hyperbolicity of $\wt\cA$ is supplied:

\begin{thmC}\label{thm:c}
Assume $\wt\cA$ consists of finitely many hyperbolic equilibria or periodic orbits, then there exists a sequence $u_{n\in\N}\in C^k(M,\R)\cap\cS_c^-$ converging to $u^-$ as $n\rightarrow+\infty$ w.r.t. the $C^0-$norm, such that for any $n\in \N$,
\begin{enumerate}
	\item  $u_n=u^-$ on $\cA$ . 
	\item $\lb u_n(x)+H(x,du_n(x))<0$ for $x\notin\cA$.
\end{enumerate} 
\end{thmC}

\medskip


\subsection{Applications of constrained subsolutions}

As the first application of constrained subsolutions, we show how to locate the {\sf maximal global attractor} of (\ref{eq:ode0}) by using elements in $\cS_c^-\cap C^{1}(M,\R)$. We will see that the differentiability plays a crucial role.

\begin{defn}
A compact $\Phi^t_{H,\lb}$-invariant set $\Om\subset T^*M$ is called a {\sf global attractor} of $\Phi^t_{H,\lb}$, if for any point $(x,p)\in T^*M$ and any open neighborhood $\cU$ of $\Om$, there exists $T(x,p,\cU)>0$ such that for all $t\geq T$, $\Phi_{H,\lb}^t(x,p)\in \cU$. Moreover, if $\Om$ is not contained in any larger global attractor, then it is called a {\sf maximal global attractor}.
\end{defn}

\begin{thmD}\label{thm:dthmD}
For any initial point $(x,p)\in T^*M$, the flow $\Phi_{H,\lb}^t(x,p)$ tends to a maximal global attractor $\cK$ as $t\rightarrow+\infty$. Moreover, $\cK$ can be identified as the forward intersectional set of the region
$
\Sigma_c^-:=\bigcap_{u\in\cS^-_c\cap C^1(M,R)}\big\{(x, p)\in T^*M\big| {\lb u(x)+H(x,p)}\leq 0\big\},$
i.e.
\[
\cK=\bigcap_{t\geq 0}\Phi_{H,\lambda}^t(\Sigma_c^-).
\]
\end{thmD}
\begin{rmk}
A similar conclusion as Theorem D was proved 
in \cite{MS}, where they used a complicated analysis to handle with the low regularity of $u^-$. Here we give a simplified poof with the help of $\cS_c^-\cap C^1(M,\R)$. 
\end{rmk}

Another application of $\cS_c^-\cap C^1(M,\R)$ is to control the convergent speed of the Lax-Oleinik semigroup, with a hyperbolic assumption of the Aubry set:
\begin{thmE}
Assume $\wt\cA$ is just a unique hyperbolic equilibrium $(x_0,0)\in TM$ with $\mu>0$ being the minimal positive Lyapunov exponent, then there exists $K>0$ which guarantees 
\be
\|\cT_t^-\mathbf{0}(x)-u^-(x)+e^{-\lb t}\alpha\|\leq Ke^{-(\mu+\lb)t},\quad\forall t\geq0.
\ee
where $\cT_t^-$ is the Lax-Oleinik semigroup operator (see (\ref{eq:evo}) for the definition) and 
\[
\alpha=\int_{-\infty}^0e^{\lb t}L(x_0,0)dt=u^-(x_0)
\]
is a definite constant.
\end{thmE}
\begin{corF}
Assume $\wt\cA$ consists of a unique hyperbolic periodic orbit, then there exists a constant $K>0$ and a constant $\mu>0$ being the minimal positive Lyapunov exponent of the hyperbolic periodic orbit, such that 
\[
\liminf_{t\rightarrow+\infty}\dfrac{\|\cT_t^-\mathbf{0}(x)-u^-(x)+e^{-\lb t}\alpha\|}{e^{-(\mu+\lb)t}}\leq K
\]
with $\alpha=u^-(x_0)$ for any $x_0\in\cA$ fixed.
\end{corF}

\begin{rmk}
For any $\psi\in C^0(M,\R)$, the convergent rate 
\be\label{eq:c0-con}
\|\cT_t^-\psi(x)-u^-(x)\|\sim O(e^{-\lb t}),\quad t\geq 0
\ee
has been easily proved in a bunch of references, e.g. \cite{DFIZ}. That implies the convergence of the Lax-Oleinik semigroup is independent of $\psi$. However, as $\lb\rightarrow 0_+$, the convergence becomes more and more  ineffective. 

On the other side, for the case $\lb=0$, The papers \cite{IM, M,WY1} have shown the exponential convergence of the Lax-Oleinik semigroup, with the assumption that Aubry set consists of finitely many hyperbolic equilibria. That gives us chance to estimate the convergence speed for $\cT_t^-{\bf 0}$, which is much better than \eqref{eq:c0-con} for $0<\lb\ll1$. So we get the Theorem E and Corollary F. \smallskip

Due to Theorem D of \cite{Man}, for generic $H(x,p)$ we can guarantee the uniqueness of the hyperbolic equilibrium (resp. periodic orbit with fixed homology class) for the associated conservative Hamiltonian equation, so the uniqueness of equilibrium (resp. periodic orbit) is not an artificial condition for \eqref{eq:ode0} with $0<\lb\ll1$. Nonetheless, we can still generalize Theorem E (resp. Corollary F) to several hyperbolic equilibriums (resp. periodic orbits), by replacing the constant $\alpha$ to a piecewise function 
\[
\alpha (x)=\int_{-\infty}^0e^{\lb t}L(z_x,0)dt,\quad x\in M
\]
where $z_x$ is an arbitrary point in the $\alpha-$limit set of the backward calibrated curve $\gamma_x^-:(-\infty,0]\to M$ ending with $x$. Unfortunately , $\alpha(x)$ is posteriorly decided by the equilibria (resp. periodic orbits) and then can't be explicitly expressed.
\end{rmk}

\vspace{20pt}

\subsection{Organization of the article}
The paper is organized as follows: In Sec \ref{s2}, we give a brief review of weak KAM theory for equation \eqref{eq:hj}, then in Sec. \ref{s2+} we give the proof of Theorem A, B and C. In Sec \ref{s3}, we discuss the global attractor and prove Theorem D. In Sec \ref{s4}, we discuss the convergent speed of the Lax-Oleinik semigroup and prove Theorem E and Corollary F. For the consistency of the proof,  some technical conclusions are moved to the Appendix.
\medskip

\noindent{\bf Acknowledgements:}  L.Jin is partly supported by by the National Key R\&D Program of China (2021YFA1001600) and the NSFC (Grant No. 12371186, 12171096). J. Zhang is supported by the National Key R\&D Program of China (No.2022YFA1007500) and the National Natural Science Foundation of China (No. 11901560, 12231010). K. Zhao  is supported by the National Natural Science Foundation of China (No. 12301233, 12171096). All the authors would like to thank the anonymous referee for the careful reading of and useful comments on the original version of this paper, which significantly helps to improve the presentation.
\vspace{20pt}

\section{Weak KAM theory of discounted systems}\label{s2}

First, we review some conclusions about the discounted Hamilton-Jacobi equation \eqref{eq:hj}.
%
%
%
%
%
Let's define the {\sf action function} by
\be\label{eq:act}
h_\lambda^t(x,y):=\inf_{\substack{\gamma\in C^{ac}([0,t],M)\\\gamma(0)=x,\ \gamma(t)=y}}\int_0^t e^{\lb s}L(\gamma(s),\dot\gamma(s))ds,\quad t\geq 0
\ee
of which the infimum $\gamma_{\min}:{[0,t]}\rightarrow M$ is always available and is $C^k-$smooth (by the {\sf Tonelli Theorem} and {\sf Weierstrass Theorem} \cite{MS}). Moreover, $(\gamma_{\min},\dot\gamma_{\min}):[a,b]\rightarrow TM$ is a solution of the {\sf Euler-Lagrange equation}
\beq\label{eq:e-l}
\left\{\tag{E-L}
\begin{aligned}
&\dot x=v,\\
&\frac d{dt}L_v(x,v)+\lb L_v(x,v)=L_x(x,v).
\end{aligned}
\right.
\eeq
For any point $(x,v)\in TM$, we denote by $\Phi_{L,\lb}^t(x,v)$ the {\sf Euler-Lagrange flow}, which satisfies $\Phi_{L,\lb}^t\circ\mathcal{L}_H=\mathcal{L}_H\circ\Phi_{H,\lb}^t$ in the valid time domain of $\Phi_{H,\lb}^t$. \\

 The {\sf backward Lax-Oleinik semigroup} operator $\cT_{t}^-: C^0(M,\R)\rightarrow C^0(M,\R)$ can be expressed by
\be\label{eq:evo}
\cT_{t}^-\psi(x):=e^{-\lambda t } \min_{y\in M } \{ \psi(y)+h^t_\lambda(y,x)  \}
\ee
which works as a viscosity solution of the following {\sf evolutionary equation}:
 \be\label{eq:evo-hj}
 \left\{
 \begin{aligned}
 &\partial_tu(x,t)+H(x,\partial_x u)+\lb u=0,\\
 &u(\cdot,0)=\psi(\cdot),\;\;\; \quad t\geq 0,\,x\in M.
 \end{aligned}
 \right.
 \ee
 As $t\rightarrow+\infty$, $\cT_t^-\psi(x)$ converges to a unique function
\be\label{eq:sta}
u^-(x):=\lim_{t\rightarrow+\infty}\cT_t^-\psi(x)=\inf_{\substack{\gamma\in C^{ac}((-\infty,0],M)\\\gamma(0)=x}}\int_{-\infty}^0e^{\lb\tau}L(\gamma,\dot\gamma)d\tau
\ee
which is exactly the viscosity solution of (\ref{eq:hj}) and proved to satisfies the following in \cite[Proposition 5,7]{MS}:

\begin{prop}\label{lem:ms}
\begin{itemize}
\item $u^-$ is Lipschitz on $M$, with the Lipschitz constant depending only on $L$ (independent of $\lb$).
\item $u^-\prec_\lb L$.
\item For every $x\in M$, there exists a {\sf backward calibrated curve} $\gamma_x^-:(-\infty,0]\rightarrow M$ which achieves the minimum of (\ref{eq:sta}). 
\item For any $t<0$,
\[
u^-(x)=e^{\lb t}u^-(\gamma_x^-(t))+\int_t^0e^{\lb s}L(\gamma_x^-(s),\dot\gamma_x^-(s))ds,
\]
  and there {is} a uniform upper bound $K$ depending only on $L$ such that $|\dot\gamma_x^-|\leq K$.
\end{itemize}
Any continuous function satisfies bullets 2 and 3 is also called a {\sf weak KAM solution} of (\ref{eq:hj}).
\end{prop}

Similarly, we can define the {\sf forward Lax-Oleinik semigroup} operator $\cT_{t}^+: C^0(M,\R)\rightarrow C^0(M,\R)$  by
\be
\cT_t^+\psi(x):=\max_{y\in M} \{e^{\lambda t} \psi(y)-h^t_\lambda(x,y)  \}
\ee
for later use.
\begin{defn}
A continuous function $f:\cU \subset \R^n \rightarrow\R$ is called {\sf semiconcave with linear modulus } if there exists $\cC>0$ such that
\be\label{eq:scl-defn}
f(x+h)+f(x-h)-2f(x)\leq \cC|h|^2
\ee
for all $x\in \cU,\ h\in\R^n$. Here { $\cC$} is called a {\sf semiconcavity constant} of $f$. Similarly we can define the {\sf semiconvex functions with linear modulus} if we change `$\leq$' to `$\geq$' in (\ref{eq:scl-defn}).  For any manifold $M$, we can always endow it by a group of compatible charts, each of which is diffeomorphic to an open set of $\R^n$. Consequently, we can generalize previous definition to $M$ with no additional difficulties. 

\end{defn}
\begin{defn}
Assume $u\in C(M,\R)$, for any $x \in M$, the closed convex set
	$$
	D^- u(x)=\Big\{p\in T^*M : \liminf_{y\rightarrow x} \frac{u(y)-u(x)-\langle p,y-x \rangle }{|y-x|} \geq 0  \Big\}
	$$
	$$
	\Big( \ \text{resp.} \ D^+ u(x)=\Big\{p\in T^*M : \limsup_{y\rightarrow x} \frac{u(y)-u(x)-\langle p,y-x \rangle }{|y-x|} \leq 0  \Big\} \Big)
	$$
	is called the {\sf sub-differential} (resp. {\sf super-differential}) set of $u$ at $x$.
\end{defn}
\begin{lem}\cite[Theorem.3.1.5]{CS}\label{D^+ convex}
$f:\cU\subset\R^d\rightarrow \R$ is semiconcave (resp. semiconvex), then $D^+f(x)$ (resp. $D^-f(x)$) is a nonempty compact convex set  for any $x\in\cU$.
\end{lem}
%
%
\begin{prop}\label{prop:domi-fun}
\begin{enumerate}
\item $\cT_{t+s}^\pm=\cT_t^\pm\circ\cT_s^\pm$;
\item if $u\leq v$, $\cT_t^\pm u\leq\cT_t^\pm v$;
\item  $u\prec_\lb L $ if and only if $u \leq \mathcal{T}^-_{t}u $, then $\cT_t^-u\prec_\lb L$.
\item  $u\prec_\lb L$ if and only if $\mathcal{T}^+_tu\leq u $, then $\cT_t^+u\prec_\lb L $.
\item For any $\psi\in C^0(M,\R)$, $\cT_t^-\psi$ is semiconcave for $t>0$. Similarly, $\cT_t^+\psi$ is semiconvex for $t>0$.
\end{enumerate}
\end{prop}
\begin{proof}
The idea is borrowed from \cite{Ber}, with necessary adaptions.

(1) For any $t, s>0$, we have
\begin{align*}
	\cT^-_{t+s}\psi (x)=&\, e^{-\lb(t+s)} \min_{y\in M} \{ \psi(y)+ h_\lambda^{t+s}(y,x)\}  \\
	=&\,  e^{-\lb(t+s)} \min_{y\in M} \{ \psi(y)+ \min_{z\in M} \{ h_\lambda^{s}(y,z)+e^{\lambda s} h_\lambda^{t}(z,x) \}\} \\
	=&\,  e^{-\lb t} \min_{y\in M}\min_{z\in M} \{ e^{-\lambda s }\psi(y)+  e^{-\lambda s  } h_\lambda^{s}(y,z)+ h_\lambda^{t}(z,x) \}        \\
	=&\,   e^{-\lb t} \min_{z\in M} \{ e^{-\lambda s } \min_{y\in M}\{\psi(y)+   h_\lambda^{s}(y,z) \}+ h_\lambda^{t}(z,x) \} \\
	=&\,   e^{-\lb t} \min_{z\in M} \{  \cT_{s}^-\psi(z) + h_\lambda^{t}(z,x) \} \\
	=&\, \cT^-_{t} \circ \cT^-_{s}\psi(x)
\end{align*}
It is similar for $\cT^+_{s+t}=\cT^+_t \circ \cT^+_s $.

(2) It is an immediate consequence of the definition of $\cT_t^-$, i.e.
$$
 \cT_t^- u=e^{-\lambda t } \min_{y\in M } \{ u(y)+h^t_\lambda(y,x)  \} \leq e^{-\lambda t } \min_{y\in M } \{ v(y)+h^t_\lambda(y,x)  \}={\cT_t^- v.}
$$
{For $\cT^+_t $ it's similar.}

(3) 
On one hand, if $u \leq \mathcal{T}^-_{t}u   $,  according to the definition of $\mathcal{T}_t^-$, we have that
\begin{align*}
  u(x)\leq \mathcal{T}_t^-u(x)  = \inf_{y\in M} \big\{e^{-\lambda t } u(y)+e^{-\lambda t} h^t_\lambda(y,x)\big\}
\end{align*}
which means that $e^{\lambda t } u(x)-u(y)\leq h^t_\lambda(y,x) $ for any $
x,y \in M$. Therefore, $u \prec_\lb L  $.

On the other hand, if $u \prec_\lb L $, we have that $u(x) \leq e^{-\lambda t } u(y)+e^{-\lambda t} h^t_\lambda(y,x) $ for any $
x, y \in M$ which implies that $u \leq \mathcal{T}^-_{t}u $ by taking the infimum of $y$. In summary, for every $t'\geq 0$, one obtains
$$
	\mathcal{T}^-_{t}u\leq    \mathcal{T}^-_{t}[ \mathcal{T}^-_{t'}u] = \mathcal{T}^-_{t+t'} u =\mathcal{T}^-_{t'} [\mathcal{T}^-_{t}u ]
$$
 which implies that $\mathcal{T}^-_{t}u \prec_\lb L $.

(4) Similar as above,  $\mathcal{T}^+_tu\leq u $ if and only if $ u\prec_\lb L  $. Hence, for every $t'\geq 0$,
$$
	\mathcal{T}^+_{t}u \geq  \mathcal{T}^+_{t}[ \mathcal{T}^+_{t'}u] = \mathcal{T}^+_{t+t'} u =\mathcal{T}^+_{t'} [\mathcal{T}^+_{t}u ]
$$
which implies that $\mathcal{T}^+_{t}u \prec_\lb L $.

(5)
 As is shown in \cite[Proposition 6.2.1]{F}, 
 for any fixed $t>0$ and $y\in M$,  $h^t_\lambda(y,\cdot)$ and $h^t_\lambda(\cdot,y)$ are both semiconcave, then $\psi(y)+h^t_\lambda(y,\cdot) $ is semiconcave and $e^{\lambda t} \psi(y)-h^t_\lambda(\cdot,y) $ is semiconvex. Notice that the semicocave (resp. semiconvex) modulus is not uniform as $t\rightarrow 0_+$, yet for any $t>0$ fixed, the modulus is a constant since $M$ is compact. Then due to  \cite[Proposition 2.1.5]{CS},  $\min_{y\in M } \big(\psi(y)+h^t_\lambda(y,x)\big)$ preserves the semiconcavity, so $\mathcal{T}^-_t \psi(x) $ is also semiconcave. Similar proof implies $\mathcal{T}^+_t \psi(x)$ is semiconvex.
\end{proof}

\vspace{20pt}

\section{Proof of  Theorem A, B and C}\label{s2+}

\subsection{Proof of Theorem A} 
 (1).
For any $\nu\in\mathfrak M_\lb$ and any $u\in\cS^-$,
\be\label{eq:mes-cal}
\int_{TM} \lb ud\nu\leq\int_{TM} \lb u^- d\nu&=&\lb \int_{TM}\int_{-\infty}^0 e^{\lb s} L(\Phi_{L,\lb}^s(x,v)) ds d\nu\nonumber\\
&=&\lb \int_{-\infty}^0 e^{\lb s}\Big(\int_{TM}L(\Phi_{L,\lb}^s(x,v))d\nu\Big) ds\\
&=&\lb \int_{-\infty}^0 e^{\lb s}\Big(\int_{TM}L(x,v)d\nu\Big) ds\nonumber\\
&=&\lb  \int_{-\infty}^0 e^{\lb s} ds\cdot \int_{TM}L(x,v)d\nu=\int_{TM}L(x,v)d\nu,\nonumber
\ee
which implies
\[
\int_{TM}L(x,v)-\lb u \ d\nu\geq 0,\quad\forall \nu\in\mathfrak M_\lb.
\]
Moreover, for any ergodic measure $\mu\in\mathfrak M_\lb$ supported in $\wt\cA$, for every $(x,v)\in supp\ \mu$, 
\ben
\int_{TM} \lb u^- d\mu&=&\lb \int_{TM}\int_{-\infty}^0 e^{\lb s} L(\Phi_{L,\lb}^s(x,v)) \ ds  d\mu\\
&=&\lb \int_{-\infty}^0 e^{\lb s} \int_{TM}L(\Phi_{L,\lb}^s(x,v)) \ d\mu ds\\
&=&\lb \int_{-\infty}^0 e^{\lb s} \int_{TM}L(x,v) \ d\mu  ds\\
&=&\lb  \int_{-\infty}^0 e^{\lb s} ds\cdot \int_{TM}L(x,v)d\mu=\int_{TM}L(x,v)d\mu.
\een
So $u^-\in\cS^-_c$.\\

(2). For any $u\in\cS_c^-$, if there is a $\mu\in\mathfrak M_\lb$ such that each step in (\ref{eq:mes-cal}) becomes an equality, then for $\mu-$a.e. $(x,v)\in TM$, we have $u=u^-$.
Therefore, for a.e. $(x,v)\in \text{supp}(\mu)$,
\ben
u(x)=u^-(x)=& \inf_{\gamma(0)=x}\int_{-\infty}^0 e^{\lb s} L(\gamma,\dot\gamma) ds\\
=& \int_{-\infty}^0 e^{\lb s} L(\Phi_{L,\lb}^s(x,v)) ds
\een
thus $\pi\Phi_{L,\lb}^t(x,v)$ is a $u^--$calibrated curve for $t\in(-\infty,0]$. Since $\mu$ is invariant, $\pi\Phi_{L,\lb}^t(x,v)$ has to be globally calibrated, i.e. $\Phi_{L,\lb}^t(x,v)\in\wt\cA$. So  $\cA(u):=$supp$(\mu)\subset\cA$.
\qed

\medskip

\subsection{Proof of Theorem B}  Due to Proposition \ref{prop:domi-fun}, as long as $t>0$ and $s>0$ sufficiently small, $\cT_s^-\cT_t^+u^-(x)$ is a subsolution of (\ref{eq:hj}). We claim that $u^-(x)= \cT_t^+u^-(x)=\cT_t^-u^-(x)=\cT_s^-\cT_t^+u^-(x) $ for any $x\in\cA $ and for any $s,t\in \R^+$. This is because  $u^-\in S_c^-$, so $u^- \prec_\lb L $ which implies that $ \cT_t^- u^- \geq u^- $ and $ \cT_t^+ u^- \leq u^- $ due to  Proposition \ref{prop:domi-fun}. For  $x\in\cA $, by the  Definition \ref{def:Aubry}, there is exist a  curve $\gamma_x^-:[0,t]\to M$ such that 
\[
e^{\lb t}u^-(\gm_x^-(t))- u^-(\gm_x^-(0))=\int_0^t e^{\lb s}L(\gamma_x^-(s),\dot\gamma_x^-(s)) \ ds
\]
 with $u^-(\gamma_x^-(0))=x$. Hence, \begin{align*}
	u^-(x) \leq \cT_t^+ u^-(x)=&\, \sup_{\substack{\gamma\in C^{ac}([0,t],M)\\\gamma(0)=x}}e^{\lb t}u^-(\gamma(t))-\int_0^te^{\lb \tau}L(\gamma(\tau),\dot\gamma(\tau))d\tau. \\
	\leq &\,  e^{\lb t}u^-(\gamma_x^-(t))+\int_0^te^{\lb \tau }L(\gamma_x^-(\tau),\dot\gamma_x^-(\tau))d\tau \\
	=&\, u^-(\gamma_x^-(0))=u^-(x) .
\end{align*}
which implies that $u^-(x)=\cT_t^+u^-(x) $ for any $x\in\cA $. Similarly we can prove $u^-(x)=\cT_t^-u^-(x) $ for any $x\in\cA$. So $\cT_s^-\cT_t^+u^-$ is indeed a solution on $\cA $. Recall that $\cT_t^+u^-(x)$ is always semiconvex, and for sufficiently small $s>0$, it is proven in following Lemma \ref{lem:convex} that $\cT_s^-\psi$ keeps the semiconvexity for any semiconvex function $\psi(x)$, then $\cT_s^-\cT_t^+u^-(x)$ is both semiconcave and semiconvex, {thus has} to be {$C^{1,1}$}. As $s$ and $t$ can be chosen arbitrarily close to zero, so we get the  desired conclusion.\qed

\begin{lem}\label{lem:convex}
Assume $H\in C^k(T^{\ast}M,\R)$ is a Tonelli Hamiltonian, for each semiconvex function $\psi:M\rightarrow\R$ with a linear modulus, there is $t_{0}>0$ such that $\cT_t^-\psi$ is semiconvex for $t\in[0,t_{0}]$.
\end{lem}

\begin{proof}
We follow the proof of \cite[Lemma 4]{Ber} to prove that, there exists $t_{0}>0$ such that, for $t\in[0,t_0], \cT_t^-\psi$ is supreme of a family of $C^{2}$ functions with a uniform $C^{2}$-bound. Then semiconvexity of $\cT_t^-\psi$ is a direct corollary of that.

\medskip

Since $\psi$ is semiconvex with a linear modulus, by \cite[Proposition 10]{Ber} or \cite[Theorem 3.4.2]{CS}, there exists a bounded subset $\Psi\subset C^{2}(M,\R)$ such that
\begin{enumerate}
  \item $\psi=\max_{\varphi\in\Psi}\varphi$,
  \item for each $x\in M$ and $p\in D^{-}\psi(x)$, there exists a function $\varphi\in\Psi$ satisfying $(\varphi(x),d\varphi(x))=(\psi(x), p)$.
\end{enumerate}
By the definition of $\cT^{-}_{t}$ and (1), we have
\begin{equation}\label{eq:sup}
\cT^{-}_{t}\psi\geq\sup_{\varphi\in\Psi}\cT^{-}_{t}\varphi.
\end{equation}

On the other hand, for the family $\Psi$, there exists $t_{0}>0$ such that, for each $t\in[0,t_0]$, the image $\cT^-_t(\Psi)$ is also a bounded subset of $C^2(M,\R)$ and for all $\varphi\in\Psi$ and $x\in M$,
\[
\cT^-_t\varphi(\gm(t))=e^{-\lb t}\varphi(x)+\int^{t}_0 e^{\lb(\tau-t)}L(\gm(\tau),\dot{\gm}(\tau))d\tau
\]
where $\gm(\tau)=\pi\circ\Phi^\tau_{H,\lb}(x,d\varphi(x))$.

Let $(\gm(\tau),p(\tau)):[0,t]\rightarrow T^{\ast}M$ be a trajectory of \eqref{eq:ode0} which is optimal for $h_\lambda^t(y,x)$, i.e., $\gm(0)=y, \gm(t)=x$ and
\[
h_\lambda^t(y,x)=\int^t_0 e^{\lb\tau}L(\gm(\tau),\dot{\gm}(\tau))d\tau.
\]
It is not difficult to see that $p(0)$ is a super-differential of the function $z\mapsto h_\lambda^t(z,x)$ at $y$. Since the function $z\mapsto e^{-\lb t}[\psi(z)+h_\lambda^t(z,x)]$ is minimal at $y$, then $p(0)\in D^{-}\psi(y)$.

\medskip

We consider a function $\varphi\in\Psi$ such that $(\varphi(y), d\varphi(y))=(\psi(y),p(0))$, then we have $(\gm(t),p(t))=\Phi_{H,\lb}^t(y,d\varphi(y))$ and
\[
\cT^-_t\varphi(x)=e^{-\lb t}\varphi(y)+\int^t_0 e^{\lb(\tau-t)}L(\gm(\tau),\dot{\gm}(\tau))d\tau=e^{-\lb t}[\psi(y)+h^t_\lambda(y,x)]=\cT^-_t\psi(x).
\]
Thus for each $x\in M$, there exists a function $\varphi\in\Psi$ such that $\cT^-_t\varphi(x)=\cT^-_t\psi(x)$, therefore
\[
\cT^{-}_{t}\psi \leq \sup_{\varphi\in\Psi}\cT^{-}_{t}\varphi.
\]
So we complete the proof.
\end{proof}

\begin{rmk}
Actually, obtaining $C^{1,1}-$functions via the forward {and} backward Lax-Oleinik semigroups  is a typical application of the Lasry-Lions regularization method. For the discounted Hamilton-Jacobi equation, the readers can refer to \cite{CCZ} for other applications of this method. Besides, in \cite{WWY3} an implicit Lax-Oleinik semigroup is the crucial weapon in exploring the global dynamics of $1^{st}$ order PDEs.
\end{rmk}

\vspace{10pt}

\subsection{ Proof of Theorem C}

To prove Theorem C, Lemma \ref{lem:ck-gra} is needed. Once it holds, then
%
$u^-$ has to be a $C^k-$graph in a small neighborhood $\cU$ of $\cA(H)$ and $(x, du^-(x))\in W^u(\wt\cA(H))$ for all $x\in\cU$.
Notice that there exists a nonnegative, $C^\infty-$smooth function $V:M\rightarrow\R$ which is zero on $\cA(H)$ and keeps positive outside $\cA(H)$. Moreover,  $\|V\|_{C^k}$ can be taken sufficiently small, so for the new Hamiltonian
\[
\wt H(x,p):=H(x,p)+V(x),
\]
 the hyperbolicity of $\wt \cA(\wt H)$ persists and  $\wt \cA(\wt H)=\wt \cA(H)$ due to the upper semicontinuity of the Aubry set (see Proposition \ref{lem:u-semi}). So if we denote by $\wt u^-$ the viscosity solution of $\wt H$,  $\wt u^-$ is also $C^k$ on $\cU$. We can easily see that $\wt u^-$ is a strict subsolution of $H$ in $\cU \backslash \cA( H) $. Outside $\cU$ we can convolute $\wt u^-$ with a $C^\infty$ function, and keeps $\wt u^-$ invariant on $\cU$. Without loss of generality, let us denote by $\wh u^-$ the modified function,
then for any $x\notin \cU$ being a differentiable point of $\wt u^-$, we have
\be
\lb \wh u^-(x)+ H(x,d \wh u^-(x))&=&\lb \wh u^-(x)+\wt H(x,d \wh u^-(x))-V(x)\nonumber\\
&\leq& \lb \wt u^-(x)+\wt H(x,d \wt u^-(x))-V(x)+\lb |\wh u^-(x)-\wt u^-(x)|\nonumber\\
& &+\max_{\theta\in[0,1]}\Big|H_p\Big(x,\theta d\wt u^-(x)+(1-\theta)d\wh u^-(x)\Big)\Big|\cdot\big|d \wt u^-(x)-d\wh u^-(x)\big|\nonumber\\
&\leq&-V(x)+ C\cdot\big[|\wh u^-(x)-\wt u^-(x)|+|d\wh u^-(x)-d\wt u^-(x)|\big]\nonumber\\
&\leq&-V(x)/2<0, \nonumber
\ee
since $|\wh u^-(x)-\wt u^-(x)|$ and $|d\wh u^-(x)-d\wt u^-(x)|$ can be made sufficiently small. Recall that $\wh u^-\big|_\cU=\wt u^-\big|_{\cU}$, so $\wh u^-$ is a $C^k$ smooth {constrained} subsolution of \eqref{eq:hj} which is a solution on $\cA(H)$ and {strict subsolution outside}.\qed

\begin{lem}[$C^k$ graph]\label{lem:ck-gra}
Assume $\wt\cA$ consists of finitely many hyperbolic equilibrium or periodic orbits, then there exists a neighborhood $V\supset\cA$, such that for all $x\in V$, $(x,d u^-(x))$ lies exactly on the local unstable manifold $W^u_{loc}(\wt\cA)$ ($W^u_{loc}(\wt\cA)$ is actually $C^{k-1}-$graphic over $M$).
\end{lem}

\begin{proof}
We claim that:\smallskip

{\tt For any neighborhood $V$ of $\cA$, there always exists an open neighborhood $U\subset V$ containing $\cA$, such that for any $x\in U$, the associated backward calibrated curve $\gamma_x^-:(-\infty,0]\rightarrow M$ would lie in $V$ for all $t\in(-\infty,0]$.} \smallskip

Otherwise, there exists a $V_*$ neighborhood of $\cA$ and a sequence $\{x_n\in V_*\}_{n\in\N}$ converging to some point $z\in\cA$, such that the associated backward calibrated curve $\gamma_n^-$ ending with $x_n$ can go outside $V_*$ for all $n\in\N$. Namely, we can find a sequence $\{T_n\geq 0\}_{n\in\N}$ such that $\gamma_n^-(-T_n)\in \partial V_*$. Due to item 2 of Proposition \ref{lem:ms}, any accumulating curve $\gamma_\infty^-$ of the sequence $\{\gamma_n^-\}_{n\in\N}$ is also a calibrated curve in 2 cases: 

Case 1: the accumulating value $T_\infty$ of associated $\gamma_\infty^-$ is finite, which implies $\gamma_\infty^-:[- T_\infty,0]\rightarrow M$ connecting $z$ and $\partial V_*$. Since $z\in\cA$ and $\wt\cA$ is $\Phi_{L,\lb}^t-$invariant, then $\gamma_\infty^-:\R\rightarrow M$ is contained in $\cA$ as well. That's a contradiction.

Case 2: the accumulating value $T_\infty$ of associated $\gamma_\infty^-$ is infinite, then $\eta_n^-(t):=\gamma_n^-(t-T_n):(-\infty,T_n]\rightarrow M$ accumulates to a $\eta_\infty^-:\R\rightarrow M$ which is globally calibrated by $u^-$. Due to the definition of $\wt\cA$, $\eta_\infty^-$ has to be contained in $\wt\cA$. That's a contradiction.

After all, the claim holds. Since $W^u_{loc}(\wt\cA)$ has to be $C^{k-1}-$graphic in a suitable neighborhood of $\wt\cA$ (Proposition B of \cite{CI}), then our claim actually indicates there exists a suitable $V\supset \cA$, such that for all $x\in V$, the backward calibrated curve $\gamma_x^-:(-\infty,0]\rightarrow M$ is unique and $(x, du^-(x))\in W^u_{loc}(\wt\cA)$.
\end{proof}
\begin{rmk}
	In \cite{Ber2}, a similar conclusion as Lemma \ref{lem:ck-gra} was discussed for conservative Hamiltonian systems. Here we use the same idea, but have to make a much more complicated analysis  to overcome the discount.
\end{rmk}

\vspace{20pt}

\section{Global attractors and the proof of Theorem D}\label{s3}

Now we use $\cS_c^-\cap C^1(M,\R)$ to identify the {\sf maximal global attractor}. Due to Theorem B, we can find a sequence in $\cS_c^-\cap C^1(M,\R)$ converging to $u^-$ w.r.t. the $C^0-$norm.\\


\noindent{\it Proof of Theorem D:}
Due to Theorem \ref{thm:0}, for any $u\in \cS_c^-\cap C^1(M,\R)$, we have $u\leq u^-$. Therefore,
\[
\{(x, p)\in T^*M| \lb u(x)+H(x,p)\leq 0\}\supset \{(x, p)\in T^*M| \lb u^-(x)+H(x,p)\leq 0\}.
\]
which accordingly indicates
\[
\Sigma_c^-=\{(x, p)\in T^*M| \lb u^-(x)+H(x,p)\leq 0\}
\]
On the other side, let us denote
 \[
 F_u(x,p):=\lb u(x)+H(x,p),
 \]
then we can prove that
\ben
\frac d{dt}\Big|_{t=0} F_u(\Phi_{H,\lb}^t(x,p))&=&\lb \langle du(x),\dot x  \rangle   + H_x(x,p) \dot x  +H_p(x,p)\dot p  \\
&=& \lb  \langle du(x),\dot x  \rangle+ H_x(x,p)  H_p(x,p)+H_p(x,p) (-H_x(x,p)-\lb p) \\
&=& \lb  \langle du(x),\dot x)   \rangle -\lambda \langle \dot x,p \rangle   \\
&\leq& \lb (H(x,du(x))+L(x,\dot x) ) -\lb \langle \dot x,p \rangle \\
&=&  \lb [ H(x,du(x))+ \langle \dot x,p \rangle-H(x,p) ] - \lambda  \langle \dot x,p \rangle      \\
&=&  -\lb [ H(x,p) - H(x,du(x)) ] \\ 
&\leq& -\lb [\lb u(x )+ H(x,p)] \\
&=& -\lb F_u(x,p)
\een
where the second equality is according to equation \eqref{eq:ode0} and the first inequality is due to Fenchel transform.
It implies that
every trajectory of (\ref{eq:ode0}) tends to $\Sigma_c^-(u)$ as $t\rightarrow+\infty$, with
\[
\Sigma_c^-(u):=\{(x, p)\in T^*M| F_u(x,p) \leq 0\}.
\]
So 
\[
\bigcap_{u\in\cS^-_c\cap C^1(M,\R)}\Sigma_c^-(u)=\Sigma_c^-
\]
is a global attracting set. Define
\[
\cK=\bigcap_{t\geq 0}\Big(\bigcap_{u\in\cS_c^-\cap C^1(M,\R)}\Phi_{H,\lb}^t(\Sigma_c^-(u))\Big),
\]
we can easily see that $\cK$ contains all the $\om-$limit sets in the phase space, then due to the definition of $\cK$, it has to be maximal.\qed
\vspace{20pt}

\section{Exponential convergence of the Lax-Oleinik semigroup}\label{s4}

\noindent{\it Proof of Theorem E:} Recall that 
\[
\cT_t^- \mathbf{0}(x)=\inf_{\substack{\gamma:[-t,0]\rightarrow M\\\gamma(0)=x}}\int_{-t}^0e^{\lb s}L(\gamma,\dot\gamma)ds,
\]
then we have 
\be\label{eq:leq}
u^-(x)-\cT_t^-\mathbf{0}(x)\geq \int_{-\infty}^{-t} e^{\lb s}L(\gamma_x^-,\dot\gamma_x^-)ds
\ee
where $\gamma_x^-:(-\infty,0]\rightarrow M$ is the backward calibrated curve by $u^-$ and ending with $x$. On the other side, suppose $\wh\gamma:[-t,0]\rightarrow M$ is the infimum achieving $\cT_t^-\mathbf{0}(x)$, then 
\be\label{eq:geq}
u^-(x)-\cT_t^-\mathbf{0}(x)&\leq& \int_{-\infty}^{-t}e^{\lb s}L(\eta,\dot\eta)ds+\int_{-t}^0 e^{\lb s}L(\wh\gamma,\dot{\wh\gamma})ds-\cT_t^-\mathbf{0}(x)\nonumber\\
&=&\int_{-\infty}^{-t}e^{\lb s}L(\eta,\dot\eta)ds
\ee
where $\eta:(-\infty,-t]\rightarrow M$ is the backward calibrated curve by $u^-$ and ending with $\wh\gamma(-t)$. Recall that $\wt\cA$ consists of a unique hyperbolic equilibrium, without loss of generality, we assume $\cA=\{z=(x_0,0)\}$. Combining (\ref{eq:leq}) and (\ref{eq:geq}) we get 
\ben
& &|u^-(x)-\cT_t^-\mathbf{0}(x)-e^{-\lb t}\alpha|\\
&\leq& \max\bigg\{\bigg|\int_{-\infty}^{-t}e^{\lb s}L(\eta,\dot\eta)ds-e^{-\lb t}\alpha\bigg|, \bigg|\int_{-\infty}^{-t} e^{\lb s}L(\gamma_x^-,\dot\gamma_x^-)ds-e^{-\lb t}\alpha\bigg|\bigg\}
\een
with 
\[
\alpha:=\int_{-\infty}^0e^{\lb t}L(x_0,0)dt=u^-(x_0).
\]
On the other side, we can claim the following conclusion:\smallskip

{\tt Claim: Suppose $\wt\cA$ consists of finitely many hyperbolic equilibria or periodic orbits. For any neighborhood $V\supset \cA$, there exists a uniform time $T_V>0$, such that for any $x\in M$, the associated backward calibrated curve $\gamma_x^-:(-\infty,0]$
$\rightarrow M$ will  stay inside $V$ for  $t\in(-\infty,-T_V)$.}\smallskip

Otherwise, there must be a neighborhood $V_*\supset\cA$ and a sequence $\{x_n\}_{n\in\N}\subset M$, such that the associated backward calibrated curve $\gamma_n^-$ would stay outside of $V_*$ at a time $-T_n$, and $T_n\rightarrow +\infty$ as $n\rightarrow +\infty$. With almost the same analysis as in Lemma \ref{lem:ck-gra}, we can show that any accumulating curve of $\{\gamma_n^-(\cdot+T_n):t\in\R\rightarrow M\}$ would be a globally calibrated curve, which implies $V_*^c\cap\wt\cA\neq\emptyset$. This contradiction leads to the claim.\smallskip

Now we choose a suitably small neighborhood $\wt\cU\supset\wt\cA=\{z\}$, such that both the Hartman Theorem is available in $\wt\cU$ and $W^u(\wt\cA)\cap \wt\cU$ is $C^{k-1}-$graphic. Due to Lemma \ref{lem:ck-gra}, there exists a constant $K_1>0$, such that for any $x\in \cU:=\pi\wt\cU$, the associated backward calibrated curve $\gamma_x^-:(-\infty,0]\rightarrow M$ can be estimated by the following:
\[
\|\gamma_x^-(-t)-z\|\leq K_1 e^{-\mu t},\quad\forall \ t\geq 0
\]
with $\mu>0$ being the largest negative Lyapunov exponent of the hyperbolic equilibrium. Due to our claim and Lemma \ref{lem:ck-gra}, there exists a constant $K_2\geq K_1$, such that for any $x\in M$,  the associated backward calibrated curve $\gamma_x^-:(-\infty,0]\rightarrow M$ satisfies
\be\label{eq:dist}
\|\gamma_x^-(-t)-z\|\leq K_2 e^{-\mu t},\quad\forall \ t\geq 0.
\ee
Due to Theorem C, we can find a sequence of $C^k$ subsolutions $\{u_n\in \cS_c^-\}_{n\in\N}$ approaching to $u^-$ w.r.t. $C^0-$norm. Then for any $x\in M$, and the associated backward calibrated curve (with a time shift) $\eta_x^-:(-\infty,-t]\rightarrow M$ ending with it, we have 
\ben
\bigg|\int_{-\infty}^{-t}e^{\lb s}L(\eta_x^-,\dot\eta_x^-)ds-e^{-\lb t}\alpha\bigg|&=&\lim_{n\rightarrow +\infty}\bigg|\int_{-\infty}^{-t}\frac{d}{ds}\Big(e^{\lb s}u_n\big(\eta_x^-(s)\big)-e^{\lb s}u_n(x_0)\Big)ds\bigg|\\
&=&\lim_{n\rightarrow +\infty}\big|e^{-\lb t}\big(u_n(\eta_x^-(-t))-u_n(x_0)\big)\big|\\
&=&\lim_{n\rightarrow +\infty}e^{-\lb t}\Big|u_n\big(\eta_x^-(-t)\big)-u_n(x_0)\Big|\\
&\leq&\lim_{n\rightarrow +\infty}e^{-\lb t} \|du_n\|\cdot \|\eta_x^-(-t)-x_0\|\\
&\leq& \lim_{n\rightarrow +\infty}e^{-\lb t} \|du_n\|\cdot K_2 e^{-\mu t}\\
&\leq & C\cdot K_2\cdot e^{-(\mu+\lb) t}
\een
where $C>0$ is an available constant due to the uniform semiconcavity of $\{u_n\}_{n\in\N}$. Taking this inequality into both (\ref{eq:leq}) and (\ref{eq:geq}), then we finish the proof.
\qed

\vspace{20pt}

\noindent{\it Proof of Corollary F:} We can totally borrow previous analysis, with only the following adaption: since now $\wt\cA$ is a periodic orbit $\{(\gamma_p(t),\dot\gamma_p(t))|t\in[0,T_p]\}$ which may no longer be an equilibrium, so we can take an arbitrary $x_0\in\cA$ and fix it. Therefore, we can only guarantee the existence of a constant $K_3>0$, such that for any $x\in M$, the associated backward calibrated curve $\gamma_x^-:(-\infty,0]\rightarrow M$  satisfies 
\[
\liminf_{t\rightarrow+\infty}\frac{\|\gamma_x^-(-t)-z\|}{e^{-\mu t}}\leq K_3,\quad\forall \ t\geq 0.
\]
That's different from (\ref{eq:dist}), but other parts of the analysis still holds.
\qed

\vspace{40pt}

\appendix

%
%

\section{Proof of Theorem \ref{thm:0}}\label{a2}

\begin{lem}\label{s ap}\cite[Lemma 2.2]{DFIZ}
Assume $G\in C(T^{\ast}M,\R)$ is fiberwise convex in $p$ and $v$ is a Lipschitz subsolution of equation
$$
G(x,d_xv)=0,\quad x\in M,
$$
then for any $\varepsilon>0$, there exists $v_{\varepsilon}\in C^{\infty}(M,\mathbb{R})$ such that $\|v-v_{\varepsilon}\|_{C^0}\leq\varepsilon$ and $H(x,d_xv)\leq\varepsilon$ for all $x\in M$.
\end{lem}

\begin{lem}\label{cp}\cite[Theorem 2.5]{DFIZ}
Assume $H\in C(T^{\ast}M,\R)$ such that $H(x,p)\rightarrow\infty$ as $|p|_{x}\rightarrow\infty$ uniformly in $x\in M$. If $v,u$ are, respectively, a sub and supersolution of \eqref{eq:hj}, then $v\leq u$ in $M$.
\end{lem}
\vspace{20pt}

\noindent{\it Proof of Theorem \ref{thm:0}:}
 Since $u^{-}$ is a supersolution of \eqref{eq:hj}, based on Lemma \ref{cp}, we obtain that
\[
u^-(x)\geq \sup_{u\in\mathcal{S}^-}u(x)\geq\sup_{u\in C^{\infty}(M,\R)\cap\mathcal{S}^-}u(x).
\]
Since $u^-$ is also a subsolution, we have that $u^-(x)\leq \sup_{u\in\mathcal{S}^-}u(x)$, then
\[
u^-(x)=\sup_{u\in\mathcal{S}^-}u(x).
\]

To prove $u^-(x)=\sup_{u\in C^{\infty}(M,\R)\cap\mathcal{S}^-}u(x) $, it suffices, for any $\delta>0$, to construct a subsolution $U_{\delta}\in C^{\infty}(M,\R)$ of \eqref{eq:hj} such that
$$
u^{-}(x)-\delta\leq U_{\delta}(x)\leq u^{-}(x)\quad\text{  on}\quad M.
$$
Set $\varepsilon=\frac{\lambda}{1+2\lambda}\cdot\delta$ and $v(x)=u^{-}(x)-\frac{1+\lambda}{1+2\lambda}\cdot\delta$,
\[
G(x,p)=\lambda\cdot u^{-}(x)+H(x,p),
\]
then $v$ is a subsolution of $G(x,d_{x}u)=0$. Due to Lemma \ref{s ap}, we obtain $v_{\varepsilon}\in C^{\infty}(M,\R)$ and set $U_{\delta}=v_{\varepsilon}$, thus
\[
u^{-}(x)-\delta = v(x)-\varepsilon\leq U_{\delta}(x)\leq v(x)+\varepsilon=u^{-}(x)-\frac{1}{1+2\lambda}\cdot\delta< u^{-}(x),
\]
and
\begin{align*}
&\lambda U_{\delta}(x)+H(x,dU_{\delta})=\lambda [U_{\delta}(x)-u^{-}(x)]+G(x,dU_{\delta})\\
=&\lambda [v_{\varepsilon}(x)-u^{-}(x)]+G(x,dv_{\varepsilon})\leq-\lambda\cdot\frac{1}{1+2\lambda}\cdot\delta+\varepsilon=0.
\end{align*}
{Since $\dt>0$ is arbitrarily chosen}, we proved the conclusion. 
\qed

\vspace{20pt}

\section{More about the Aubry set}\label{a3}

\vspace{20pt}
In \cite{MS}, their Aubry set is defined by
\[
\wt\cA_1:=\bigcap_{t\geq 0}\Phi_{L,\lb}^{-t}(\wt\Sigma_{L,\lb})
\]
where 
\[
\wt\Sigma_{L,\lb}:=\{(\gamma(0),\lim_{t\rightarrow 0_-}\dot\gamma(t))\in TM|\gamma\in C^{ac}((-\infty,0],M)  \text{ is backward calibrated}\}.
\]
Actually, $\wt\cA_1$ is equal to the following 
\[
\wt\cA_2:=\bigcup_{\gamma \in C^{ac}(\R,M)}\,\,\{(\gamma, \dot\gamma)|\gamma \text{ is globally calibrated by }u^-\}\subset TM,
\]
Here is the proof:
Due to conclusion (a.3) in Sec. 4 of \cite{CS},  $\wt\cA_1 \subset \wt\cA_2$. 
On the other hand, any curve $\gamma :\R\rightarrow M$ globally calibrated by $u^-$ is definitely backward calibrated by $u^-$, so $(\gamma(t),\dot\gamma(t))\in\wt\Sigma_{L,\lb}$ for all $t\in\R$. Furthermore, $(\gamma,\dot\gamma)\in \Phi_{L,\lambda}^{-t}(\wt\Sigma_{L,\lb} )$ for all $t\geq 0$. That implies $(\gamma,\dot\gamma)\in\wt\cA_1$, so $\wt\cA_2\subset\wt\cA_1$ follows.\medskip

To show the equivalence between $\wt\cA_2$ and our definition in (\ref{defn:aub}), it suffices
to prove any curve $\gamma$ globally calibrated by some $u\in\cS^-$ has to be globally calibrated by $u^-$ too. 
Due to Definition \ref{def:Aubry}, for any $a<b\in \R$, we have 
$$
e^{\lb b}u(\gamma(b))-e^{\lb a} u(\gm(a))=\int_a^b e^{\lb t}L(\gamma(t),\dot\gamma(t))dt,
$$
which leads to 
\[
e^{\lb b}u(\gamma(b))=\int_{-\infty}^b e^{\lb t}L(\gamma(t),\dot\gamma(t))dt
\]
by making $a\rightarrow-\infty$. Due to Theorem \ref{thm:0} and (\ref{eq:sta}), 
\[
u(\gamma(b))\leq u^-(\gamma(b))=\inf_{\substack{\eta\in C^{ac}((-\infty,0],M)\\\eta(0)=\gamma(b)}}\int_{-\infty}^0e^{\lb\tau}L(\eta,\dot\eta)d\tau.
\]
Combining these conclusions we get $u(\gamma(b))=u^-(\gamma(b))$ for all $b\in\R$.
So $\gamma:\R\rightarrow M$ is globally calibrated by $u^-$. Finally we get $\wt\cA_2=\wt\cA$.

\begin{prop}[Upper semicontinuity]\label{lem:u-semi}
As a set-valued function,  
\[
 (0,1]\times C^k(TM,\R)\ni (\lb,L)\longrightarrow \wt\cA(L)\subset TM
\] is upper semicontinuous.
\end{prop}
\begin{proof}
First, we prove any accumulating curve of $\{\gamma_n\in\cA(\lb_n)\}$ is contained in  $\cA(\lb_*)$ for $\lb_n\rightarrow\lb_*$ and a fixed $L$. Due to item 2 of Proposition \ref{lem:ms}, for any $\lb_n\in(0,1]$, $\wt\cA(\lb_n)$ is uniformly compact in the phase space. Therefore, for any sequence $\{(\gamma_n,\dot\gamma_n)\}_{n\in\N}$ globally calibrated, the accumulating $(\gamma_*,\dot\gamma_*)$ has to satisfy
\[
\int_t^se^{\lb_*\tau}L(\gamma_*,\dot\gamma_*)d\tau= \lim_{n\rightarrow+\infty}\int_t^se^{\lb_n\tau}L(\gamma_n,\dot\gamma_n)d\tau=\lim_{n\rightarrow+\infty}\int_t^se^{\lb_*\tau}L(\gamma_n,\dot\gamma_n)d\tau
\]
for any $t<s\in\R$. On the other side, for any $\eta\in C^{ac}([t,s],M)$ satisfying $\gamma_*(s)=\eta(s)$ and $\gamma_*(t)=\eta(t)$, we can find a list of enlarging intervals $[t_n,s_n]$ such that 
\[
\lim_{n\rightarrow+\infty}t_n=t \;\text{(resp. $\lim_{n\rightarrow+\infty}s_n=s$)}.
\]
Then a sequence $\eta_n\in C^{ac}([t_n,s_n],M)$ uniformly converging to $\eta$ (as $n\rightarrow+\infty$) is achievable, and we can request that  $\gamma_n(s_n)=\eta_n(s_n)$ and $\gamma_n(t_n)=\eta_n(t_n)$ for all $n\in\N$. Recall that 
\[
\int_{t_n}^{s_n}e^{\lb_n\tau}L(\eta_n,\dot\eta_n)d\tau\geq \int_{t_n}^{s_n}e^{\lb_n\tau}L(\gamma_n,\dot\gamma_n)d\tau,
\]
so we get 
\ben
\int_t^se^{\lb_*\tau}L(\gamma_*,\dot\gamma_*)d\tau&=& \lim_{n\rightarrow+\infty}\int_t^se^{\lb_n\tau}L(\gamma_n,\dot\gamma_n)d\tau\\
&= &\lim_{n\rightarrow+\infty}\int_{t_n}^{s_n}e^{\lb_n\tau}L(\gamma_n,\dot\gamma_n)d\tau\\
&\leq & \lim_{n\rightarrow+\infty}\int_{t_n}^{s_n}e^{\lb_n\tau}L(\eta_n,\dot\eta_n)d\tau\\
&=&\int_t^se^{\lb_*\tau}L(\eta,\dot\eta)d\tau,
\een
which indicates $\gamma_*:\R\rightarrow M$ minimizes {$h_{\lb_*}^{s-t}(\gamma_*(t),\gamma_*(s))$} for any $t<s\in\R$. So  $\gamma_*\in\cA(\lb_*)$.

For fixed $\lb\in (0,1]$ and $L_n\rightarrow L_*$ w.r.t. the $C^k-$norm, the proof is similar.
\end{proof}



\vspace{35pt}

\bibliographystyle{plainurl}
\bibliography{smooth}
\end{document}